\documentclass[12pt,a4paper]{amsart}
\usepackage{a4wide} 

\usepackage{pdfpages}
 
 \usepackage{amsmath, amssymb, amsfonts, enumerate}
\usepackage[all]{xy}
\usepackage{amscd}
\usepackage{comment}
\usepackage{mathtools}
\usepackage{tikz} 
\usepackage{tikz-cd}

\usepackage{url}
\usepackage{hyperref}

\newtheorem{theoremletter}{Theorem}

\makeatletter
\newtheorem*{rep@theorem}{\rep@title}
\newcommand{\newreptheorem}[2]{%
\newenvironment{rep#1}[1]{%
 \def\rep@title{#2 \ref{##1}}%
 \begin{rep@theorem}}%
 {\end{rep@theorem}}}
\makeatother

\newtheorem{theorem}{Theorem}[section]
\newtheorem*{theorem*}{Theorem}
 \newtheorem*{conjecture*}{Conjecture}
  
  \newtheorem*{corollary*}{Corollary}
  
  \newtheorem{lemma}[theorem]{Lemma}
\newtheorem{corollary}[theorem]{Corollary}
\newtheorem{proposition}[theorem]{Proposition}

\newreptheorem{theorem}{Theorem}
 \newreptheorem{proposition}{Proposition}
 \newreptheorem{corollary}{Corollary} 

 \theoremstyle{definition}
 \newtheorem{definition}[theorem]{Definition} 

 \newtheorem{remark}[theorem]{Remark}

 \newtheorem{example}[theorem]{Example}

\numberwithin{equation}{section}
 
\newcommand {\N}{\mathbb{N}} 
\newcommand {\Z}{\mathbb{Z}} 
\newcommand {\R}{\mathbb{R}} 
 
\newcommand {\C}{\mathbb{C}} 

\newcommand{\G}{\mathbb{G}}

\newcommand{\DD}{\mathcal{D}}

\newcommand{\OO}{\mathcal{O}}

\newcommand{\Proj}{\mathbb{P}}

\DeclareMathOperator{\vol}{vol}

\DeclareMathOperator{\Ker}{Ker}

\DeclareMathOperator{\im}{Im}

\DeclareMathOperator{\rank}{rank}

\DeclareMathOperator{\mult}{mult} 

\DeclareMathOperator{\Tr}{Tr} 
 
\DeclareMathOperator{\length}{length}

\DeclareMathOperator{\divisor}{div} 
 
\DeclareMathOperator{\val}{val}

\usepackage{setspace}
\begin{document}	
\title[Large unions of generalized integral points on elliptic surfaces]
{Large unions of generalized integral sections on elliptic surfaces}
\author[Xuan-Kien Phung]{Xuan Kien Phung}
\address{Universit\'e de Strasbourg, CNRS, IRMA UMR 7501, F-67000 Strasbourg, France}
\email{phung@math.unistra.fr}
\subjclass[2010]{}

\begin{abstract}   
Let $f \colon X \to B$ be a nonisotrivial complex elliptic surface and let $\mathcal{D} \subset X$ 
be an integral divisor dominating $B$. 
We study finiteness related properties of generalized 
$(S, \mathcal{D})$-integral sections $\sigma \colon B \to X$ of $X$. 
These integral sections $\sigma$ correspond to rational points in $A(K)$ which satisfy   
the set-theoretic condition $f ( \sigma(B) \cap \mathcal{D})\subset S$, 
where $S \subset B$ is an arbitrary given subset. For $S \subset B$ finite, 
the set of $(S, \mathcal{D})$-integral sections of $X$ is finite 
by the well-known Siegel theorem.  
In this article, we establish a general quantitative finiteness result 
of several large unions of $(S, \mathcal{D})$-integral sections 
in which both the subset $S$ and the divisor $\mathcal{D}$ are allowed to vary 
in families where notably $S$ is not necessarily finite nor countable. 
Some applications to generalized unit equations over function fields are also given.   
\end{abstract}

\maketitle
 
\setcounter{tocdepth}{1}
\tableofcontents


\section{Introduction}

\label{chapter:strong-uniform-elliptic}


\textbf{Notations.}  
We   fix throughout a compact Riemann surface $B$ of genus $g$. 
Denote $K= \C(B)$ its function field. 
For every complex space $X$, 
the pseudo Kobayashi hyperbolic metric on $X$ is 
denoted by the symbol $d_{X}$ (cf. Definition \ref{pseudo Kobayashi hyperbolic metric}). 
 Discs in Riemannian surfaces are assumed to 
have sufficiently piecewise smooth boundary. 
The symbol $\#$ stands for cardinality.

\subsection{Generalized integral points}
 
In this paper, we investigate finiteness related properties of 
certain unions  of \emph{generalized integral points} in elliptic 
curves over function fields. The general definition 
is as follows. 

\begin{definition} 
[$(S, \DD)$-integral point and section] 
\label{d:s-d-integral-point-geometric}
Let $f \colon  X \to B$ be a proper flat morphism. 
Let $S \subset B$ be a subset (not necessarily finite) 
and let $\DD \subset X$ be a subset (not necessarily an effective divisor). 
A section $\sigma \colon B \to X$ is said to be $(S, \DD)$\emph{-integral}
 if it satisfies the set-theoretic condition (cf. Figure \ref{fig:integral-section}): 
\begin{equation}
\label{e:s-d-integral-point-geometric}
f(\sigma(B) \cap \DD) \subset S. 
\end{equation}
For every $P \in X_K(K)$, let $\sigma_P \colon B \to X$ 
be the corresponding section.  
Then $P$ is called an $(S, \DD)$\emph{-integral point} 
of $X_K$  
if the section $\sigma_P$ is $(S, \DD)$-integral. 
\end{definition} 

  \begin{figure}[ht]
  \centering
  \includegraphics[page=1,height=.25\textwidth, width=.52\textwidth]{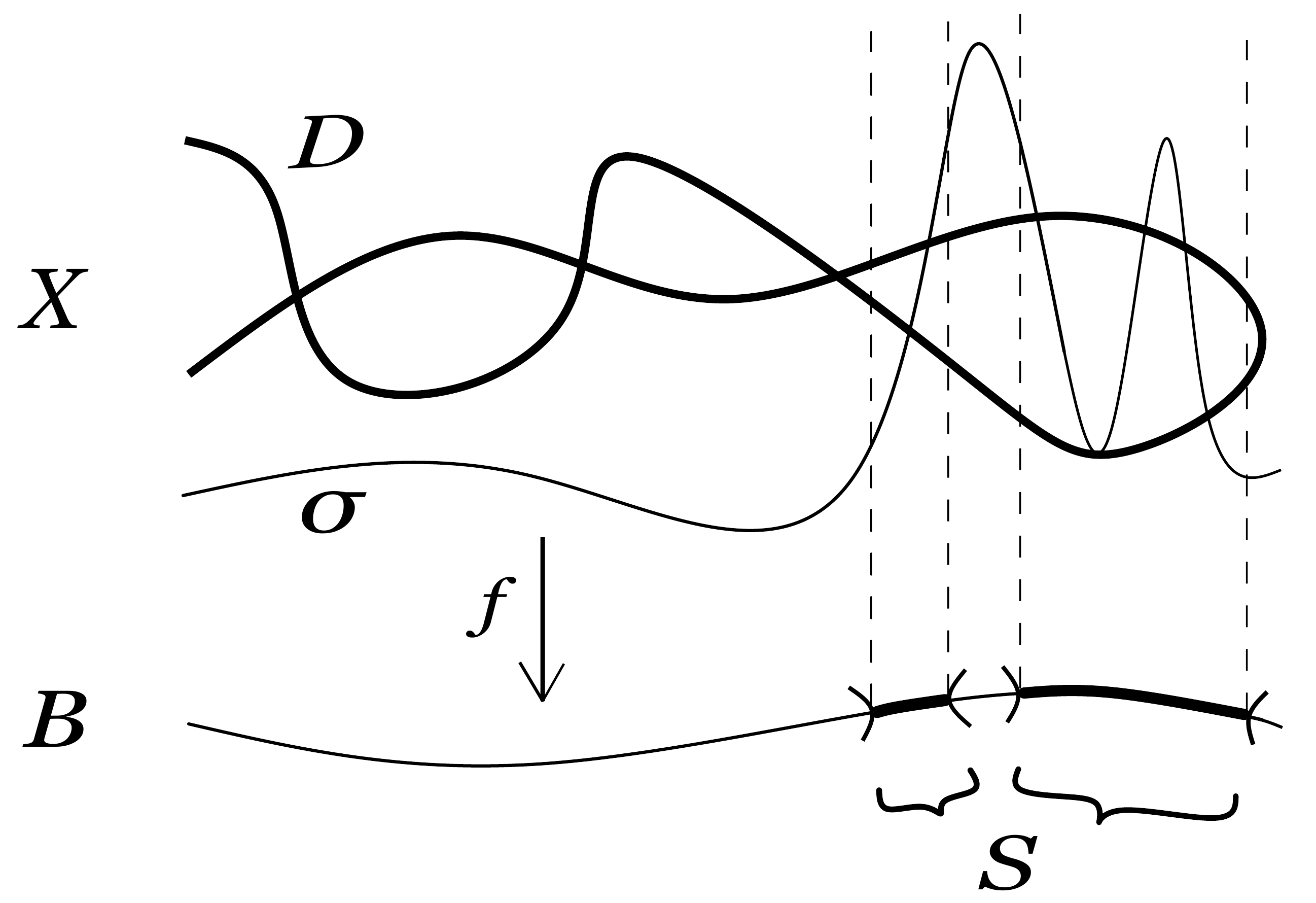}\hspace*{.0\textwidth}%
\caption{An $(S, \DD)$-integral section $\sigma$}
    \label{fig:integral-section}
\end{figure} 

The above notion of generalized integral points and sections are introduced 
and studied in a more general context of families of abelian varieties in \cite{phung-19-abelian}. 
It is not hard to see that Definition \ref{d:s-d-integral-point-geometric} generalizes 
the notion of integral solutions of a system of Diophantine equations 
when the subsets $S\subset B$ are finite 
(cf., for example, \cite[Introduction]{phung-19-phd}). 
\par

\par
Let $d$ be a Riemannian metric on $B$, the following 
general quantitative result is shown in \cite{phung-19-abelian} 
which generalizes a theorem of Parshin in \cite{parshin-90}.

\begin{theorem} 
\label{t:parshin-generic-emptiness-higher-dimension} 
Let $A / K$ be an abelian variety with a model $f \colon \mathcal{A} \to B$. 
Let $\DD \subset \mathcal{A}$ be the Zariski closure of an effective ample divisor $D \subset A$. 
Assume that $D$ does not contain any translate of nonzero abelian subvarieties of $A$.   
Let  $\varepsilon >0$. 
There exists a finite union of disjoint closed discs $Z_\varepsilon \subset B$ with $\vol_d Z_\varepsilon < \varepsilon$ 
satisfying the following. 
Let $W $ be a finite union of disjoint closed discs in $B$ disjoint to $Z_\epsilon$, let 
$B_0= B \setminus (W\cup Z_\varepsilon)$.  
There exists $m >0$ such that: 
\begin{enumerate} [\rm (*)] 
\item
For $I_s$ ($s \in \N$) the union of  $(S, \DD)$-integral points of $A$   
over all subsets $S \subset B$ such that  $\# \left(S \cap B_0 \right) \leq s$, we have: 
\begin{equation}
\label{e:qualitative-bound-abelian}
\# I_s  < m(s+1)^{2 \dim A . \rank \pi_1(B_0)}, \quad \text{ for every } s \in \N. 
\end{equation} 
\end{enumerate}
\end{theorem}

Here, a model of $A$ means a N\' eron model or a proper flat morphism $\mathcal{A} \to B$ 
such that $\mathcal{A}_K\simeq A$ and $\rank \pi_1(B_0) \in \N$ denotes the minimal number of generators of 
$\pi_1(B_0)$. 
\par
Let the notations be as in Theorem \ref{t:parshin-generic-emptiness-higher-dimension}. 
The trace $\Tr_{ k(\bar{C} )/ k}(A)(\C)$ of $A$ (cf. \cite{chow-image-trace-1}, \cite{chow-image-trace-2})  acts naturally on $A(K) = \mathcal{A}(B)$ by translations. 
If $A/K$ is a nonisotrivial elliptic  curve,   its trace is zero and the condition 
$D$ not containing any translate of nonzero abelian subvarieties is empty.  
In this case, 
Theorem \ref{t:parshin-generic-emptiness-higher-dimension} shows  
that the union $I_s$ of $(S, \DD)$-integral points is finite and grows at most as a polynomial 
in terms of $s$. 
We remark moreover that the growth order $2 \dim A.\rank \pi_1(B_0)$ in Theorem \ref{t:parshin-generic-emptiness-higher-dimension} 
 is a reasonably optimal order that one can possibly expect in general (cf. \cite[Remark 1.10]{phung-19-abelian}).  
\par
Classical approaches via height theory 
can establish the finiteness as well as the polynomial growth of $I_s$ only when 
the set $W\cup Z_\varepsilon$ in Theorem \ref{t:parshin-generic-emptiness-higher-dimension} 
is finite. This is the content of the generalized Siegel theorem over function fields (cf. \cite{lang-60}) for 
integral points of bounded denominators (cf., for example, \cite[Corollary 1.7]{phung-19-parshin-integral}, 
\cite[2.9]{shioda-schutt-lecture}). 
While the intersection height theory works over an arbitrary field (cf. \cite{fulton-98}), 
the best we can obtain with this tool is a uniform bound on the intersection multiplicities 
(as in \cite{buium-94}, \cite[Theorem 0.6]{hindry-silverman-88}, \cite{noguchi-winkelmann-04}). 
For this reason, whenever $S$ is infinite, the height with respect to the divisor $\DD$ of an $(S, \DD)$-integral point, 
which is the sum of the intersection multiplicities over $S$, cannot be simply bounded 
using the uniform bound on the intersection multiplicities. 
\subsection{Main result} 
In this article, we   develop the technique studied in \cite{phung-19-abelian}  
in the case of elliptic curves. 
It turns out that we can obtain in this situation  
a very strong property on the finiteness of certain unions $J_s$ "twice larger" than $I_s$ (cf. \eqref{e:qualitative-bound-abelian})
consisting of $(S, \DD)$-integral points in which both the set $S$ and the divisor  $\DD$ are allowed to vary in families. 
Moreover, it is shown that the growth of $J_s$ in terms of $s$ is still 
at most polynomial of degree $2   \pi_1(B_0)$ as in Theorem \ref{t:parshin-generic-emptiness-higher-dimension}.   
Recall that an effective divisor $D$ on a fibered variety $X$ over $B$ is called \emph{horizontal} 
if the induced map $D \to B$ is dominant. 
Otherwise, $D$ is said to be \emph{vertical}. 
\par
Our setting is as follows. 

\subsection*{Setting (E)} 
Fix a nonisotrivial elliptic surface $f \colon X \to B$.
 Denote by  $T \subset B$   the \emph{type} of $X$, i.e., the finite subset above
which the fibres of $f$ are not smooth. 
Let  $\tilde{Z}$ be a  smooth complex  algebraic variety. 
Let $ {\DD} \subset X \times \tilde{Z}$ be an algebraic family of relative horizontal effective Cartier  
divisors. Assume that $\DD \to B \times \tilde{Z}$ is flat. 
Let $Z \subset \tilde{Z}$ be a relatively compact subset  
with respect to the complex topology.  
\par
The first main result of the article is the following (cf. Section \ref{s:strong-uniform-elliptic}). 
  
\begin{theoremletter}
\label{t:strong-uniform-elliptic} 
Let the notations and hypothesis be as in Setting (E). 
Consider any finite union of disjoint closed discs
$V \subset B$ containing $T$ such that distinct points in $T$ are contained in different discs. 
For each $s \in \N$, the following union of integral points 
\begin{align*}
J_s \coloneqq \cup_{z \in Z} \cup_{S \subset B, \# (S \setminus V) \leq s} 
\{(S, \DD_z)\text{-integral points of } X_K  \} \subset X_K(K)
\end{align*}
is finite. Moreover, there exists $m > 0$ such that for every $s \in \N$, we have:
\begin{equation*}
\label{e:qualitative-bound-elliptic-strong}
\#J_s \leq m(s+1)^{2\rank \pi_1(B \setminus V)}. 
\end{equation*}
 \end{theoremletter} 

\begin{remark}
In fact, the exact same proof of Theorem \ref{t:strong-uniform-elliptic} presented in this article shows that 
the conclusion of Theorem \ref{t:strong-uniform-elliptic} also holds when $X \to B$ is an isotrivial 
elliptic surface, up to replacing $\# J_s$ by $\# (J_s \text{ mod } \Tr_{K / \C}(X_K)(\C))$.   
\end{remark}

The proof of Theorem \ref{t:strong-uniform-elliptic} 
is a combination of the hyperbolic-homotopic method developed in \cite{phung-19-abelian} 
with a technical Lemma \ref{l:key-lemma-strong-uniform-elliptic} which controls locally 
the hyperbolic metric on certain smaller subsets of $X \setminus \DD_z$ when $z$ varies. 
In particular, our proof does not use the usual height theory and thus no height bound is established.

\begin{remark}
\label{r:miracle-flaness}
By the miracle flatness theorem (cf. \cite[Lemma 00R4]{stack-project}, \cite[Theorem 23.1]{matsumura-crt}), 
the condition requiring $\DD \to B \times \tilde{Z}$ to be flat in Setting (E) is equivalent 
to the condition saying that for every $z \in \tilde{Z}$, the divisor $\DD_z$ contains no vertical components 
(or also by \cite[Proposition 3.9]{liu-alg-geom} since $B$ is an algebraic curve).  
\end{remark}

Theorem \ref{t:strong-uniform-elliptic} can be seen as a certain generalization 
of the following theorem of Hindry-Silverman (cf. \cite[Theorem 0.6]{hindry-silverman-88})) which 
is obtained using the   N\' eron-Tate height theory. 

\begin{theorem*}
[Hindry-Silverman] 
Let $(O)$ be the zero section in $X$ and 
$r= \rank X_K(K)$. 
There exists an explicit function $c(g)>0$ such that:  
\begin{equation*} 
\# \{(S, (O))\text{-integral points of }X_K\} \leq c(g)((\#S)^{1/2} +1)^r. 
\end{equation*} 
\end{theorem*} 

The above result of Hindry-Silverman holds when $B$ is defined over any algebraically closed field $k$ 
of characteristic $0$. But in the case $k= \C$, our Theorem \ref{t:strong-uniform-elliptic} 
gives a much stronger finiteness result  
without losing the polynomial bound of reasonable degree order. 

\par
To conclude, some applications to generalized integral sections on ruled surfaces 
(cf. Theorem \ref{t:s-unit-equation-geometric-general})  
and to generalized unit equations over function fields (cf. \eqref{e:unit-equation-geometric-general},  
Corollary \ref{c:s-unit-equation-geometric-general-1}, 
Corollary \ref{c:few-solution-generalized-unit}) are 
also given in Section \ref{s:unit-equation-hyperbolic}.

\section{The hyperbolic-homotopic height machinery} 

We collect and formulate in this section the tools that are necessary for the article.

\subsection{The homotopy reduction step of Parshin}  

We continue with the notations as in Setting (E).  
We can suppose without loss of generality that $X_K[n] \subset X_K(K)$ for some integer $n \geq 2$. 
 Let $U$ be a disjoint union of closed discs in $B$ 
such that any two distinct points in $T$ are contained in different discs. 
Let $b_0 \in B_0 \coloneqq B \setminus U$. 
Denote $\Gamma = H_1(X_{b_0}, \Z)$ and $G= \pi_1(B \setminus U, b_0)$. 
\par 
Since $X_{B_0} \to B_0$ is a proper submersion, 
it is a fibre bundle by Ehresmann's fibration theorem (cf. \cite{ehresmann}). 
It follows that we have an exact sequence of fundamental groups 
induced by the fibre bundle $ X_{b_0} \to X_{B_0} \to B_0$ of $K(\pi, 1)$-spaces:  
\begin{equation}
\label{e:abelian-homotopy-exact-sequence-1}
0 \to \pi_1(X_{b_0}, w_0)=H_1(X_{b_0}, \Z) \to \pi_1(X_{B_0}, w_0) \xrightarrow{\rho \,= f_*} \pi_1(B_0,b_0) \to 0. 
\end{equation}

To fix the ideas, $w_0$ is chosen here and in the rest of the article to be the 
zero point of $X_{b_0}$, which also lies on the zero section of $X_{B_0}$.    
\par
We  fix a collection of smooth geodesics  $l_{w_0,w} \colon [0,1] \to X_{b_0}$ 
on the torus $X_{b_0}$ 
such that $l_{w_0,w} (0)=w_0$ and $l_{w_0,w} (1)=w$. 
Every (analytic or algebraic) section $\sigma_P \colon B_0 \to X_{B_0}$ 
gives rise to a section $i_P \colon \pi_1(B_0,b_0) \to \pi_1(X_{B_0}, w_0)$ of the exact sequence 
\eqref{e:abelian-homotopy-exact-sequence-1} as follows. 
Given any loop $\gamma$ of $B_0$ based at $b_0$, 
we define the section $i_P$ by the following formula: 
\begin{equation}
\label{e:definition-of-i-p-main-reduction-step}
i_P([\gamma])=[ l^{-1}_{w_0,\sigma_P(b_0)} \circ \sigma_P(\gamma) \circ l_{w_0,\sigma_P(b_0)}] 
\in  \pi_1(X_{B_0}, w_0). 
\end{equation} 

Note that as oppose to the   composition of homotopy classes, 
 we  concatenate oriented paths as maps as above,  
so the order reverses. 
The quantitative version 
of the homotopy reduction step of Parshin can now be stated 
in our context as follows: 
 
\begin{proposition}
\label{p:homotopy-rational-abelian}
Let the notations be as in Setting (E) and let $U \subset B$ 
be a disjoint union of closed discs in $B$ 
such that any two distinct points in $T$ are contained in different discs.   
 Every section $i$ of the exact sequence 
\eqref{e:abelian-homotopy-exact-sequence-1} 
is induced by at most $t_X=\#  X_K(K)_{tors} $ rational points $P \in X_K(K)$ 
as in the definition \eqref{e:definition-of-i-p-main-reduction-step}. 
\end{proposition}

\begin{proof}  
cf. \cite[Proposition 9.2]{phung-19-abelian}, see also \cite[Proposition 2.1]{parshin-90}. 
More details and description can be found in \cite{phung-19-phd}.  
\end{proof}

\subsection{Hyperbolic height on Riemann surfaces} 

One of the key ingredients in the proof of 
Theorem \ref{t:parshin-generic-emptiness-higher-dimension} 
in \cite{phung-19-abelian}  
is the following linear bound 
on the hyperbolic length of loops in 
various complements of a Riemann surface. 
Let $U$ be any finite union of disjoint closed discs in the Riemann surface 
$B$ and denote $B_0 \coloneqq B \setminus U$. 
It is shown in \cite{phung-19-abelian} that: 
 
\begin{theorem}
\label{t:linear-bound-s-base-curve-1}
For every free homotopy class $\alpha \in \pi_1(B_0)$, 
there exists $L >0$ with the following property. 
For any finite subset $S \subset B_0$,  
there exists  a piecewise smooth loop  
$\gamma \subset B_0 \setminus S$ which represents     
the free homotopy class $\alpha$ in $B_0$ and satisfies:  
\begin{equation}
\label{e:main-hyper-length-bound-intro} 
\length_{d_{B_0 \setminus S}} (\gamma) \leq L(\#S + 1). 
\end{equation}
\end{theorem} 

Recall that $d_{B_0 \setminus S}$ denotes the intrinsic Kobayashi hyperbolic metric on $B_0 \setminus S$ 
(cf. Definition \ref{pseudo Kobayashi hyperbolic metric})

\subsection{Metric properties of hyperbolic manifolds} 

In this section, we collect some fundamental properties of the 
pseudo Kobayashi hyperbolic metric  
and of hyperbolic manifolds due to Green. 
Let $X$ be a complex manifold. 
The \emph{pseudo Kobayashi hyperbolic metric} $d_X \colon X \times X \to X$ 
is defined as follows. 
Let $\rho$ be the Poincar\' e metric on the unit disc $\Delta= \{z \in \C \colon |z|=1\}$.  
\par
Let $x, y \in X$. Consider the data $L$ consisting of 
a finite sequence of points $x_0=x, x_1, \dots, x_n=y$ in $X$,  
a sequence of holomorphic maps 
$f_i \colon \Delta \to X$ and of pairs $(a_i, b_i) \in \Delta^2$ for 
$i=0, \dots, n$ such that $f_i(a_i)=x_i$ and $f(b_i)=x_{i+1}$. 
Let $H(x,y; L)= \sum_{i=0}^n \rho(a_i, b_i)$. 

\begin{definition}
[cf. \cite{kobay-hyperbolic-definition}]
\label{pseudo Kobayashi hyperbolic metric} 
For $x, y \in X$, we define $d_X(x,y) \coloneqq \inf_{L}   H(x,y;L)$. 
\end{definition}

 $X$ is called a \emph{hyperbolic} manifold if 
 $d_X(x,y)>0$ for all distinct points $x, y \in X$, i.e., 
when $d_X$ is a metric.  
The fundamental distance-decreasing property 
(cf. \cite[Proposition 3.1.6]{kobay-hyperbolic}) can be stated as follows:  

\begin{lemma}
\label{l:distance-decreasing-hyperbolic}
Let $f \colon X \to Y$ be a holomorphic map of complex manifolds. 
Then for all $x, y \in X$, $d_Y(f(x), f(y)) \leq d_X(x,y)$. 
In particular, if $X \subset Y$,     $d_Y\vert_X \leq d_X$. 
\end{lemma}
 
A complex space $X$ is said to be \emph{Brody 
hyperbolic} if it does not contains entire curves, i.e., 
there is no nonconstant holomorphic maps 
$\C \to X$.

\begin{theorem} [Green]
\label{t:green-hyperbolic-embedding} 
Let $X$ be a relatively compact open subset of a complex manifold $M$.  
Let $D \subset X$ be a closed complex subspace. 
Denote by $\bar{X}, \bar{D}$ the closures of $X$ and $D$ in $M$. 
Assume that $\bar{D}$ and $\bar{X} \setminus \bar{D}$ are Brody hyperbolic. 
\par
Then $X \setminus D$ is hyperbolic and we have     
$d_{X \setminus D} \geq \rho\vert_{X \setminus D}$ 
for some Hermitian metric $\rho$ on $M$. 
In particular, if $M$ is compact and $\lambda$ is any Riemannian metric on $|M|$ then 
there exists $c >0$ such that 
$d_{X \setminus D} \geq c \lambda \vert_{X \setminus D}$. 
\end{theorem}

\begin{proof}
See \cite[Theorem 3]{gre-78}. 
\end{proof}

\begin{remark} 
\label{r:infinitesimal-hyperbolic}
Let $X$ be a smooth complex manifold. 
Let $ \Delta(0,R) = \{z \in \C \colon |z| < R\}$ for every $R >0$. 
The infinitesimal Kobayashi-Royden pseudo metric $\lambda_X$ on $X$ 
corresponding to the  
Kobayashi pseudo hyperbolic metric $d_X$  
can be defined as follows. 
For $x \in X$ and every vector $v \in T_xX$,  
$\lambda_X(x, v) \coloneqq \inf 2/R$, 
where the minimum is taken over all $R >0$ for which there exists a holomorphic map 
$f \colon \Delta(0,R) \to X$ such that $f'(0)=v$. 
\end{remark}
 
Thanks to the distance-decreasing property of the pseudo-Kobayashi hyperbolic metric, 
we have the following important property of sections. 

\begin{lemma}
\label{l:section-geodesic}
Let $f \colon X \to Y$ be a holomorphic map between complex spaces. 
Suppose that $\sigma \colon Y \to X$ is a holomorphic section. 
Then $\sigma(Y)$ is a totally geodesic subspace of $X$, i.e., for all 
$x, y \in Y$, we have 
$d_Y(x, y)= d_{X} (\sigma(x), \sigma(y))$. 
\end{lemma}

\begin{proof}
By Lemma \ref{l:distance-decreasing-hyperbolic}, 
$d_Y (x, y) = d_Y (f(\sigma(x)), f(\sigma(y))) \leq d_{X} (\sigma(x), \sigma(y)) \leq d_Y (x, y)$. 
\end{proof}

\section{A key technical lemma}
\label{s:key-lemma} 
For the proof of Theorem \ref{t:strong-uniform-elliptic},  
we fix a Hermitian metric $\rho$ on the smooth surface $X$. 
It is clear that we can assume $\tilde{Z}$ integral. 
Define $B_0 \coloneqq B \setminus V$ and fix a system generators 
$\alpha_1, \dots, \alpha_k$ consisting of free homotopy classes of the fundamental groupoid $\pi_1(B_0)$. 
For a complex space $Y$, the symbol 
 $d_Y$ always denotes the Kobayashi hyperbolic pseudo-metric
on $Y$ (cf. Definition \ref{pseudo Kobayashi hyperbolic metric}).
\par
The proof of Theorem \ref{t:strong-uniform-elliptic} is based on the following key technical lemma 
which is of local nature. 
The main idea is the following. 
For each $z \in \tilde{Z}$, the hyperbolic metric on 
$(X \setminus \DD_z)\vert_{B_0}$ 
dominates the Hermitian metric $\rho$ by 
Green's theorem \ref{t:green-hyperbolic-embedding} 
up to a certain strictly positive factor. 
When $z$ varies in $\tilde{Z}$, these factors vary as well and may 
\emph{a priori} tend to $0$. 
The point of Lemma \ref{l:key-lemma-strong-uniform-elliptic} below is that 
for $z$ in a small neighborhood of $\tilde{Z}$, we can, up to restricting   
further $(X \setminus \DD_z)\vert_{B_0}$ over some fixed 
nice complement of $B_0$ (cf. Property (b) below), these factors are in fact bounded below by a strictly 
positive constant (cf. Property (P) below).   
 
 \begin{lemma} 
\label{l:key-lemma-strong-uniform-elliptic} 
Let the notations be as in Theorem \ref{t:strong-uniform-elliptic}. 
Let $\varepsilon > 0$. 
Then there exists $ M > 0$ such that for each $z_i \in \tilde{Z}$, 
we have the following data:
\begin{enumerate} [\rm (a)]
\item
an analytic open neighborhood $U_i$ of $z_i$ in $\tilde{Z}$;
\item
a disjoint union $V_i \subset B_0$ consisting of $\leq M$ closed discs each of radius $\leq \varepsilon$; 
\item
a constant $c_i > 0$;
\end{enumerate}
with the following property: 
\begin{enumerate} [\rm (P)]
\item 
for each $z \in U_i$, we have 
$
d_{(X \setminus \DD_z) \vert_{(B_0 \setminus V_i)}} \geq c_i \rho\vert_{(X \setminus \DD_z)\vert_{(B_0 \setminus V_i)}}. 
$
 \end{enumerate}

\end{lemma}

We remark first a standard lemma. 

\begin{lemma}
\label{l:analytic-cover-deformation} 
Let $f \colon X \to Y$ be a proper flat morphism of integral   
complex algebraic varieties 
of the same dimension. 
Assume that $Y$ is a smooth manifold  
 and  for some $y \in Y$, the fibre $X_y$ is reduced and finite.  
Then there exists an analytic open neighborhood $U \subset Y$ 
of $y$ such that $f^{-1}(U) \to U$ is a finite \' etale cover. 
\end{lemma}

\begin{proof}
Since $X_y$ is finite and reduced and $f$ is flat, 
every point $x \in X_y$ is a smooth point of $f$ since we are in characteristic $0$ 
(cf. \cite[Lemma 3.20]{liu-alg-geom}). 
Since $Y$ is regular, 
we deduce from the open property of smooth morphisms (cf. \cite[Definition 6.14]{gorz})  
that every point of $X_y$ is a regular point of $X$ (cf. \cite[Theorem 4.3.36]{liu-alg-geom}). 
Since $\dim X= \dim Y$ and $Y$ is a smooth manifold, 
\cite[Definition 6.14]{gorz}  implies that the map $f$ is submersive at each point of the fibre $X_y$. 
Let $S \subset X$ be the set of singular points of $X$ (i.e., where $X$ is not locally a manifold) 
and let $S' \subset X$ be the set where     $f$ is not submersive. 
Since $f$ is proper,   $A= f(S) \cup f(S')$ is a closed subset of $Y$. 
Then by the proof of \cite[Theorem 21, Chapter III]{gunning-rossi}, 
$X \setminus f^{-1}(A) \to Y \setminus A$ is an \' etale cover. 
We have seen that $X_y \cap (S \cup S')= \varnothing$. 
Therefore, $y \notin A$ and  there exists 
an analytic open neighborhood $U \subset Y \setminus A$ 
of $y$ such that $f^{-1}(U) \to U$ is a finite \' etale cover. 
 \end{proof}

For ease of reading, 
we mention here the following key theorem 
of Brody: 

\begin{theorem} [Brody's reparametrization lemma]
\label{l:brody-reparametrization}
Let $M$ be a complex manifold with (possibly empty) boundary. 
Let $H$ be a Hermitian metric on $M$. 
Suppose that $f \colon \Delta_R \to M$ be a holomorphic 
map with $|df(0)|_H >c$ form some $c>0$. 
Then there exists a holomorphic map 
$g \colon \Delta_R \to M$ satisfying the following conditions: 
\begin{enumerate} [\rm (a)]
\item
$|dg(0)|_H=c$;
\item
$|dg(z)|_H \leq \frac{c R^2}{R^2-|z|^2}$ for all $z \in \Delta_R$;
\item
$\im (g) \subset \im (f)$. 
\end{enumerate}
\end{theorem}

\begin{proof}
See Brody's reparametrization lemma, page 616 in \cite{gre-78} or \cite{brody-78}.  
 \end{proof}

For the proof of Lemma \ref{l:key-lemma-strong-uniform-elliptic}, 
we claim first that there exists $N > 0$ such that the total
number of irreducible components (counted with multiplicities) of each effective divisor $\DD_z$ 
is at most $N$ for every $z \in  \tilde{Z}$. 
Indeed, let $H$ be any ample divisor on $X$ then $C \cdot H \geq 1$ for
every irreducible curve $C \subset X$. 
Remark that the divisors $\DD_z$ are all algebraic equivalent 
(since $\tilde{Z}$ is integral hence connected by curves) thus numerically equivalent. 
Therefore, $N \coloneqq \DD_z \cdot H$ 
is a constant independent of $z \in  \tilde{Z}$. 
By the linearity of the intersection pairing, the
above two remarks clearly show that the total number of irreducible components (counted
with multiplicities) of $\DD_z$ is at most $N$ as claimed. 
\par
For each $z \in \tilde{Z}$, notice that the effective divisor  
$\DD_z$ contains only horizontal components 
with respect to the fibration $f \colon X \to B$ by 
Remark \ref{r:miracle-flaness}. 
We   denote by $(\DD_z)_{red}$ 
the induced reduced scheme structure of $\DD_z$. 
By the adjunction formula, we find that 
$$
p_1 \coloneqq p_a(\DD_z) = \frac{\DD_z(\DD_z+K_X)}{2} +1 \geq 0
$$
is a constant independent of $z \in \tilde{Z}$. 
Now, since the arithmetic genus of $\DD_z$ are uniformrly 
bounded, a version of the Riemann-Hurwitz theorem (cf. \cite[Propositions 7.4.16, 7.5.4]{liu-alg-geom}) 
for the ramified cover
of algebraic curves $\pi_z \colon (\DD_z)_{red} \to  B$ implies that for all $z \in  \tilde{Z}$:
$$
2(p_1-2) +2N \geq \# \{ \text{ramification points of } \pi_z \} + 
\# \{ \text{singular points of } (\DD_z)_{red} \}. 
$$
\par
Let $T_z \subset B$ be the image in $B$ of the union of the ramification points of $\pi_z$ and of 
  the singular points of $(\DD_z)_{red}$. 
It follows that $T_z$ is finite and we have: 
\begin{equation}
\label{e:key-lemma-strong-uniform-elliptic-1} 
\# T_z  \leq M\coloneqq  2(p_1 -2)+2N, \quad \text{ for all } z \in \tilde{Z}. 
\end{equation}
\par
We can now return to the proof of  Lemma \ref{l:key-lemma-strong-uniform-elliptic}.

\begin{proof}[Proof of Lemma \ref{l:key-lemma-strong-uniform-elliptic}] 
Let us fix $\varepsilon >0$ sufficiently small and $z_i \in \tilde{Z}$. 
Clearly, we can choose a finite disjoint union $V_i$ of at most $M$ (defined in \eqref{e:key-lemma-strong-uniform-elliptic-1}) 
nonempty open discs in $B_0$ of $\rho$-radius $\leq \varepsilon$ 
to cover the points of $T_{z_i} \setminus  V$. 
Recall that $T_{z_i} \subset B$ is the image in $B$ of the union of 
  the ramification points of 
$\pi_{z_i} \colon (\DD_{z_i})_{red} \to B$ and of 
  the singular points of $(\DD_{z_i})_{red}$. 
 This gives the data $V_i$ for (b).
\par
Denote $Y_z \coloneqq (X \setminus \DD_z)\vert_{B_0 \setminus V_i}$ 
for each $z\in \tilde{Z}$. 
To show the existence of $U_i$ and $c_i$ so that (i)
is satisfied, we suppose the contrary. 
Hence, by the continuity of $\rho$ on the compact unit
tangent space of $X$, there would exist a sequence $(z_n)_{n \geq 1} \subset \tilde{Z}$ 
such that $z_n \to z_i$ in the
analytic topology and a sequence of holomorphic maps
$$
h_n \colon  \Delta_{R_n} \to  Y_{z_n}
$$
such that $R_n \to \infty$ 
and $|dh_n(0)|_\rho > 1$. 
Here, $\Delta_{R} \subset \C$ denotes the open disc of radius $R$ in the complex plane.  
By Brody's reparametrization lemma (Theorem \ref{l:brody-reparametrization}), we obtain a sequence of holomorphic maps
$$
g_n \colon \Delta_{R_n} \to Y_{z_n} 
$$
such that $|dg_n(0)|_\rho = 1$ and $|dg_n(z)|_\rho \leq 
 R^2_n/ (R_n^2 - |z|^2)$ 
for all $z \in \Delta_{R_n}$. 
In particular, $|dg_n(z)|_\rho \leq  4/3$ for all $z \in \Delta_{R_n/2}$. 
It follows that the family $(g_n)_{n\geq 1}$ is equicontinuous
with image inside the compact space $X\vert_{B_0 \setminus V_i}$. 
By the Arzela-Ascoli theorem, we deduce, up to passing to a subsequence, 
that $(g_n)_{n \geq 1}$ converges uniformly on compact subsets of $\C$ to
a map $g \colon \C \to  X\vert_{B_0 \setminus V_i}$. 
A standard argument using Cauchy's theorem and Morera's theorem shows that $g$ is a holomorphic map
and we also have $|dg(0)|_{\rho} = 1$. 
\par
Consider the holomorphic composition map 
$\pi \circ g \colon \C \to  B_0 \setminus  V_i$. 
Since $B_0 \setminus  V_i$ is hyperbolic,
 it is Brody hyperbolic and thus 
the map $\pi \circ g$ is a constant $b_* \in B_0 \setminus  V_i$. 
Remark that $B \times \tilde{Z}$ is smooth and $\DD \to B \times \tilde{Z}$ is a proper flat morphism of relative dimension $0$ by hypotheses.  
Since moreover the fibre of $\DD_{red}$ over $(b_*, z_i)$ is reduced and finite, 
we can apply Lemma \ref{l:analytic-cover-deformation}. 
It follows that there exists a constant $\lambda \in \N$, 
a small analytic open disc $\Delta \subset B_0 \setminus V_i$ containing $b_*$ 
and a small analytic connected open neighborhood $U_i$ of $z_i$ in $\tilde{Z}$ 
such that $\DD\vert_{\Delta \times U_i} \to \Delta \times U_i$ 
is an \' etale cover of degree $\lambda$. 
Up to shrinking $\Delta$ and $U_i$, we can suppose that 
$\DD\vert_{\Delta \times U_i}$ consists of $\lambda$ disjoint connected components.  
In particular, $\DD_z\vert_{\Delta} \to \Delta$ 
is  
an \' etale cover consisting of disjoint 
$\lambda$-sheets for every $z \in U_i$. 
 
Fix a connected component $\DD^0$ of $\DD \vert_{\Delta \times U_i}$. 
We thus obtain a family of $1$-sheeted covers $D_{z}^0 \to \Delta$ with $z \in U_i$.  
Each $D_{z}^0$ is a complex submanifold of $X_\Delta$ via the inclusion $\iota_z \colon D_{z}^0 \to X_\Delta$. 
The composition $\pi \circ \iota_z \colon D_{z}^0 \to \Delta$ is thus holomorphic and bijective for every $z \in U_i$.  
Since every injective holomorphic map is biholomorphic to its image (cf. \cite[Theorem 2.14, Chapter I]{range-86-book}),  
we deduce that the map $s_{z} \colon \Delta \to D_{z}^0 \to X_\Delta$ given by $t \mapsto D_{z}^0(t)= \pi^{-1}(t) \cap D_{z}^0$ 
is holomorphic for every $z \in U_i$. 
 
\par
We can write $X_\Delta \subset \Proj^2 \times \Delta$ as defined by the Weierstrass equation: 
$$
y^2= x^3 + A(t) x + B(t)
$$
where $A(t), B(t)$ are holomorphic functions on $\Delta$ such that the discriminant 
$4A^3+27B^2$ does not vanish on $\Delta$ (since $\Delta \cap T= \varnothing$ where we recall that 
$T \subset B$ is the finite subset above
which the fibres of $f$ are not smooth).  
\par
We can thus write $s_z(t)=(u(t, z), v(t, z))$ where $u, v \colon \Delta \times U_i \to \C$ are holomorphic functions. 
Since the translation maps on the elliptic fibration $X_\Delta$ are algebraic 
thus holomorphic,  
the maps $\Psi_{z} \colon X_\Delta \to X_\Delta$ given by the translations by $s_{z} - s_{z_i}$:  
$$
\Psi_z(x)= x + s_{z} (f(x)) - s_{z_i} (f(x)), \quad x \in X_\Delta, 
$$ 
form a smooth family of biholomorphisms commuting with the map $f \colon X_\Delta \to \Delta$. 
 \par
 Since $g_n \to g$ uniformly on compact subsets of $\C$ and $\im(g) \subset f^{-1}(b_*)$, 
we can, up to passing to a subsequence with suitable restrictions of domains of definitions, suppose
that $\im(g_n) \subset X_\Delta$ 
and that the holomorphic maps $g_n \colon \Delta_{R_n} \to X_\Delta \setminus \DD_{z_n}$ still satisfy the
properties: 
$$
R_n \to \infty , \quad |dg_n(0)|_\rho = 1. 
$$
Consider now the sequence of holomorphic maps
$$
f_n \coloneqq \psi_{z_n} \circ g_n \colon \Delta_{R_n} \to  X_\Delta \setminus \DD_{z_i}
$$
into the fixed space $X_\Delta \setminus \DD_{z_i}$. 
Then by the smoothness of the family of biholomorphisms
$(\psi_z)_{z \in U_i}$ 
 and the compactness of the $\rho$-unit tangent bundle of $X$, 
there exists a constant $c > 1$ such that 
$$
c^{-1}\leq |df_n(0)|_\rho \leq c, \quad \text{ for all } n \geq 1. 
$$
Since $R_n \to \infty$, Remark \ref{r:infinitesimal-hyperbolic} then implies 
immediately a contradiction to the fact that 
$X_{\Delta} \setminus \DD_{z_i}$ is hyperbolically
embedded in $X_{\Delta}=f^{-1}(\Delta)$ 
(for the Hermitian metric $\rho\vert_{f^{-1}(\Delta)}$) 
by Green's theorem \ref{t:green-hyperbolic-embedding}. 
To check the latter fact, it suffices to remark that   
$\bar{\Delta}$, thus  $\DD_{z_i}\vert_{\bar{\Delta}}$ are Brody hyperbolic and 
moreover, the fibres of $X_{\bar{\Delta}} \setminus (\DD_{z_i}\vert_{\bar{\Delta}})$ are 
also Brody hyperbolic.  
\par
Hence, the existence of the data in (a), (b) and (c) such that (P) is satisfied.
\end{proof}

\section{Proof of the main result}   
\label{s:strong-uniform-elliptic}

We can now return to the proof of Theorem \ref{t:strong-uniform-elliptic}.

\begin{proof} [Proof of Theorem \ref{t:strong-uniform-elliptic}] 
We have defined $B_0 = B \setminus V$ at the beginning of Section \ref{s:key-lemma} 
and we have fixed a system of simple generators 
$\alpha_1, \dots, \alpha_k$ of the fundamental group $\pi_1(B_0)$ 
 with fixed based point. 
Recall that $Z$ is a compact subset of $\tilde{Z}$ with respect to the complex 
topology. 
\par
Fix $\varepsilon > 0$ sufficiently small. 
We can enlarge slightly the discs in $V$ 
if necessary. 
Then we obtain a constant $M > 0$ 
and a set of data $(U_i,  V_i,  c_i)$ 
for each element $z_i \in Z$ as in Lemma \ref{l:key-lemma-strong-uniform-elliptic}. 
We have obviously an open covering $Z \subset \cup_{z_i \in Z} U_i$ of $Z$. 
Since $Z$ is compact, there exists a finite subset 
$Z_* \subset Z$ such that $Z \subset \cup_{z_i \in Z_*} U_i$. 
As $c_i > 0$ for every $z_i \in Z$ by Lemma \ref{l:key-lemma-strong-uniform-elliptic}, 
we can define 
\begin{equation}
\label{e:c-star-elliptic-strong-proof-main}
c_* \coloneqq \min_{z_i \in Z_*} c_i > 0. 
\end{equation}
Denote $W_i \coloneqq  V \cup V_i \subset B$ and $B_i \coloneqq B \setminus W_i= B_0 \setminus V_i$ for each $z_i \in Z_*$. 
Since the $V_i$'s are union of disjoint closed discs, they are contractible. 
Hence, we have a canonical inclusion $\pi_1(B_0) \subset \pi_1(B_i)$ for each $z_i \in Z_*$. 
\par
For every $z_i \in Z_*$, we denote by $L_i>0$ 
the maximum of the constants given by Theorem \ref{t:linear-bound-s-base-curve-1} applied to the 
the compact Riemann surface $B$ and 
the disjoint union of closed discs  $W_i$ and to each free homotopy classes  
$\alpha_1, \dots, \alpha_k$ viewed as elements of $\pi_1(B_i)$. 
We define: 
\begin{equation}
\label{e:l-max-strong-elliptic-proof-main}
L= \max_{z_i \in \Z_*} L_i \in \R_+. 
\end{equation}
\par
Now let $P \in J_s$, i.e., $P \in X_K(K)$ 
is an $(S, \DD_z)$-integral point for some $z \in Z$ 
and for some $S \subset B$ 
such that $\# (S\cap B_0 )\leq s$. 
Since $Z  \subset \cup_{z_i \in Z_*} U_i$, there exists $z_i \in Z_*$ 
such that $z \in U_i$. 
By the definition of the constant $L_i$ and by Theorem \ref{t:linear-bound-s-base-curve-1} applied 
to $B_i $, 
there exists $b_i \in B_i$ and a system of loops 
$\gamma_1, \dots,  \gamma_k$ based at $b_i$
 representing respectively the homotopy classes $\alpha_1, \dots, \alpha_k$ 
 up to a single conjugation 
such that  $\gamma_j \subset B_i \setminus S$ and that 
\begin{equation}
\label{e:strong-uniform-elliptic-proof-main} 
\length_{d_{B_i \setminus S}} (\gamma_j) 
\leq L_i(s + 1).
\end{equation} 

Now, let $\sigma_P \colon B \to X$ 
be the corresponding section of the rational point $P$. 
For every $j \in \{1, \dots, k\}$, we have
$$
\sigma_P(\gamma_j) \subset 
 (X \setminus \DD_z)\vert_{B_i \setminus S} \subset (X \setminus \DD_z)\vert_{B_i} . 
$$
This is true because $P$ 
is $(S, \DD_z)$-integral so that $\sigma_P(\gamma_j)$ 
cannot intersect $\DD_z$ outside of $f^{-1}(S)$ 
and because  
$\gamma_j \subset B_i \setminus S$. 
It follows that:

\begin{align*}
\length_\rho (\sigma_P(\gamma_j)) 
& \leq c_i^{-1} \length_{d_{(X \setminus \DD_z)\vert_{B_0 \setminus V_i}}} 
(\sigma_P (\gamma_j)) 
&  \text{(by Lemma \ref{l:key-lemma-strong-uniform-elliptic}})
\\
& \leq 
 c_i^{-1} \length_{d_{(X \setminus \DD_z)\vert_{B_i \setminus S}}} 
(\sigma_P (\gamma_j)) 
& \text{(as } ( X \setminus \DD_z)\vert_{B_i \setminus S} \subset (X \setminus \DD_z)\vert_{B_i})  
\\
& \leq 
c_*^{-1} \length_{d_{B_i\setminus S}} (\gamma_j) 
& (\text{by } \eqref{e:c-star-elliptic-strong-proof-main} \text{ and Lemma } \ref{l:section-geodesic})
\\
& \leq  c_*^{-1}L_i( s+1) 
& (\text{by }  \eqref{e:strong-uniform-elliptic-proof-main}) 
\\
& \leq  c_*^{-1}L (s+1).  
&   (\text{by } \eqref{e:l-max-strong-elliptic-proof-main}) 
\end{align*} 
Remark that the second inequality in the above follows from Lemma \ref{l:distance-decreasing-hyperbolic} 
and from the definition $B_i= B_0 \setminus V_0 \supset B_i \setminus S$.  
\par
Now let $\delta$ be the diameter of $X_{B_0\cup \partial B_0}$ 
with respect to the metric $\rho$. 
Remark  that by hypothesis, distinct points of the non smooth locus $T \subset B$ 
of $f$ are contained in distinct discs of $V$. 
It follows that the
homotopy section $i_P$ (cf. \eqref{e:definition-of-i-p-main-reduction-step})
associated to the rational point $P$ of the short exact sequence \eqref{e:abelian-homotopy-exact-sequence-1} 
$$
0 \to  \pi_1(X_{b_0}, w_0) \to \pi_1(X_{B_0}, w_0) \to \pi_1(B_0, b_0)\to 0
$$
sends the basis $(\alpha_j)_{1 \leq j \leq k}$ of $\pi_1(B_0, b_0)$ to the homotopy  classes in $\pi_1(X_{B_0}, w_0)$ 
which admit representative loops of $\rho$-lengths bounded by 
$H(s) \coloneqq c_*^{-1}L(s+1) +  2\delta$. 
The constant $2 \delta$ in the definition of $H(s)$ 
corresponds 
to the bound on the length of the extra paths induced by the change of base points 
from $\sigma_P(b_i)$ to $w_0=\sigma_O(b_0)$.  
\par
From the above bound $H(s)$ on the length of the image loops by $(S, \DD_z)$-integral sections 
for every $z \in Z$ and every $S \subset B$ such that $\# (S\cap B_0) \leq s$, 
the rest of the proof follows exactly the same lines  
as in the proof of \cite[Theorem A]{phung-19-abelian} which 
uses the homotopy reduction Proposition \ref{p:homotopy-rational-abelian} and 
the geometry of the fundamental groups $\pi_1(X_{b}, w_b)$ for 
$b \in B_0$ and $w_b \in X_b$ as a counting lemma 
(cf. \cite[Lemma 13.10]{phung-19-abelian}).   
Therefore, we obtain  a constant $m >0$ (independent of $s$) such that: 
\begin{equation*}
\#J_s \leq m(s+1)^{2 \rank \pi_1(B_0)}, \quad \text{for every } s \in \N. 
\end{equation*}
\end{proof}

\section{Applications to generalized unit equations over function fields}
\label{s:unit-equation-hyperbolic}

\subsection{Statement of the main result}
The goal of the present section is to apply the approach used to prove 
Theorem \ref{t:strong-uniform-elliptic} 
to obtain similar results in the case of ruled surfaces. 
Throughout this section, we will fix a compact connected Riemann surface $B$ 
of function fields $K= \C(B)$, and a finite subset $S \subset B$. 
\par 
Moreover, 
the following definitions and notations are used: 
\begin{enumerate} [\rm (1)]
\item
$B_S \coloneqq B \setminus U$ where $U$ is a finite  disjoint union of closed discs centered at points of $S$;  
\item
$B_0 \coloneqq B_S \setminus V$ where $V \subset B_S$ is a finite disjoint union of closed  discs;  
\item
$X = \Proj^1_\C \times B$ and $f \colon X \to B$ is the second projection;    
\item 
$\rho$ is a fixed Hermitian metric on $X$;  
\item
$\sigma_x \colon B \to X$ denotes the section induced by   $x \in K$ and let $(x)\coloneqq \sigma_x(B)$; 
\item 
$F = \G_m = \Proj^1 \setminus \{0, \infty\}$  is the fibre of $f$ and  
$(0), (\infty) \subset X$ are the constant sections associated to 
the points $0, \infty \in \Proj^1(K)$;
\item
$Y= F \times B_0 = f^{-1} ( B_0 ) \setminus \left( (0) \cup (\infty) \right) \subset X_0 = \Proj^1 \times B_0$. 
\end{enumerate}

\begin{remark}
\label{r:section-is-unit}
The set of sections of $F \times (B \setminus S) \to B \setminus S$ 
is canonically identified with the set $\OO_S^*$ of $S$-units of $K$. 
Indeed, sections of the surface $F \times (B \setminus S)$ are 
exactly sections of $X$ which do not intersect $(0)$ and $(\infty)$ at points lying above $B \setminus S$. 
On the other hand, each element $x \in K^*$ corresponds canonically 
to a non zero section denoted $\sigma_x \subset X$ and vice versa by the valuative critera for properness. 
The condition $x \in \OO_S$ (resp. $x^{-1} \in \OO_S$) means exactly that $\sigma_x(B)$ does not intersect $(\infty)$ (resp. $(0)$) 
at points lying above $B \setminus S$. 
Therefore, $x \in \OO_S^*$ if and only if $\sigma_x$ is a section 
of $F \times (B \setminus S) \to B \setminus S$ as claimed.    
\end{remark}

\begin{definition}
Let $D \subset X$ be an effective divisor 
and let $R \subset B$ be a subset. 
We say that a point $x \in K^*$ 
is \emph{$(R, D)$-integral} if it satisfies 
$$
f(\sigma_x(B) \cap D) \subset R.
$$ 
\end{definition}

\begin{definition} 
\label{d:integral-ruled-surface}
For each subset $R \subset B$, 
the \emph{generalized ring of $R$-integers} of $K$ is denoted by 
$\OO_R = \{ x \in K\colon \val_\nu (x) \geq 0, \, \forall \nu \in B \setminus R \} \subset K$. 
\end{definition}

\begin{definition}
For each $\varepsilon \geq 0$, we define 
the \emph{$(B_0, \varepsilon)$-interior subset} of $\OO_S^*$ by 
$$
\OO_{S}^*(B_0, \varepsilon) \coloneqq 
\left \{ x \in \OO_S^* \colon |x(t)|, |x(t)^{-1}| > \varepsilon, \, \forall t \in B_0 \right \}. 
$$ 
\end{definition}
Equivalently, the image of every meromorphic function $x \in \OO_{S}^*(B_0, \varepsilon)$ on $B$ does not 
meet the $\varepsilon$-neighborhoods of $0$ and $\infty$ in $\Proj^1$. 
Remark that $\OO_S^*(B_0, 0)= \OO_S^*$ since $0, \infty \notin B_0$.  
Moreover, we have $\OO_S^*= \cup _{\varepsilon >0} \OO_S^* (B_0, \varepsilon)$. 
 \par 
The main result of this section is the 
following quantitative finiteness result for large unions of generalized integral points on rational curves 
over function fields. 
 
\begin{theoremletter}
\label{t:s-unit-equation-geometric-general} 
Let $\tilde{Z}$ be a smooth complex algebraic variety 
and $Z \subset \tilde{Z}$  a compact subset with respect to the complex topology.     
Suppose that $\DD \subset X \times \tilde{Z}$ is a family of effective divisors   
such that $\DD \to B \times \tilde{Z}$ is flat and $\DD_z$ is not contained in $(0) \cup (\infty)$ 
for every $z \in \tilde{Z}$. 
For every $r \in \N$ and $\varepsilon >0$, the following union of integral points: 
\begin{align*}
J_{r, \varepsilon} \coloneqq \cup_{z \in Z}   \cup_{R \subset B, \# R \cap B_0 \leq r}  
\{ x \in \OO_{S}^*(B_0, \varepsilon) \colon x \text{ is } (R,\DD_z)\text{-integral in } X   \} \subset \OO_S^*, 
\end{align*}
is finite modulo $\C^*$. Moreover, there exists a constant $m >0$ such that:   
\begin{align*}
\label{t:s-unit-equation-geometric-general-conclusion}
\# (J_{r, \varepsilon} \text{ mod } \C^*) \leq m(r + 1)^{2 \rank\pi_1(B_0)} \text{ for every } r \in \N.  
\end{align*}
\end{theoremletter}

The proof of Theorem \ref{t:s-unit-equation-geometric-general} given in Section 
\ref{s:proof-of-ruled-surface} 
 will follow closely the steps in the proof of Theorem \ref{t:strong-uniform-elliptic}. 
It might be helpful to first consider the following counter-example explaining 
why we cannot take the whole set of $S$-units $\OO_S^*$, i.e., $\varepsilon = 0$, 
in the union $J_{r, \varepsilon}$.  

\begin{example}
Let the notations and hypotheses be as in Theorem \ref{t:s-unit-equation-geometric-general}. 
Assume moreover that $B = \C\Proj^1$ is the Riemann sphere and 
that $S= \{0, \infty\}$. 
Let $t$ be the inhomogenous coordinate on $\Proj^1$ then 
$\OO_S^* = \C^*.\{ t^n \colon n \in \Z\} = \{c t^n \colon c \in \C^*, n \in \Z\}$. 
Suppose also that $Z= \tilde{Z}= \{\cdot\}$ and 
$\DD= (1) \subset X$ is the section induced by $1 \in K$. 
\par
For every $n \in \Z$, we define $c_n = \left(2 \sup_{t \in B_0} |t^n|\right)^{-1}$. 
Then $c_n >0$ and is a finite number since $B_0 \subset \C\Proj^1$ is a complement of a finite union $U \cup V$ of closed discs with nonempty interior 
containing $0$ 	and $\infty$.  
It follows that $0 < |c_nt^n| < 1$ for all $t \in B_0$. 
Therefore, for each $n \in \Z$, $x_n \coloneqq c_n t^n \in \OO_S^*$ is   
$(R, \DD)$-integral  in $X$ in the sense of Definition \ref{d:integral-ruled-surface} 
where $R= U \cup V$. 
In particular, $\{ x_n \colon n \in \Z\} \subset J_{r,0}$ for all $n \in \Z$ and $r \in \N$. 
However, $\{ x_n \colon n \in \Z\} $ modulo $\C^*$ is  $\Z$ and thus $J_{r, 0}$ is infinite. 
Remark that there exists $t_0 \in B_0 \subset \C$ such that $|t_0| \neq 0, 1$.  
Hence, if $|t_0 >1$, we see that $c_n \to 0$ when $n \to - \infty$. 
Otherwise, if $0 < |t_0 | <1$ then $c_n \to 0$ as $n \to  + \infty$. 
\par
Therefore,  it is necessary to restrict to the union of integral points $J_{r, \varepsilon}$ 
where $\varepsilon >0$ to obtain a finiteness result as in Theorem  
\ref{t:s-unit-equation-geometric-general}. 
\end{example}

 \subsection{Some applications to generalized unit equations}
 
We illustrate in this section several applications of Theorem \ref{t:s-unit-equation-geometric-general} 
in a context generalizing the classical $S$-unit equations 
(cf. \cite{evertse-86}, \cite{evertse-gyory-88}, \cite{evertse-gyory-stewart-tijdeman}). 
 \par
 Let  $ D_S \coloneqq \sum_{ b \in S} [b]$ 
 be the effective divisor of $B$ associated to the finite subset $S \subset B$.  
Let $r \in \N$, consider the following subset of $K$: 
\begin{equation}
\label{e:definition-o-b-r}
\begin{aligned}
\OO_{B_0, r} 
& \coloneqq \cup_{R \subset B_0, \# R \leq r}  \{x \in K \colon \val_\nu(x) \geq 0, \text{ for every } \nu \notin R \cup B_0 \} \\ 
& = \cup_{R \subset B_0, \# R \leq r} \OO_R. 
\end{aligned}
\end{equation}

Remark that $\OO_{B_0, r} \subset \OO_{B_0, r+1}$ for every $r \geq 0$ and 
$K= \cup_{r \geq 0} \OO_{B_0, r}$. 
Moreover, $\OO_S \subset \OO_{B_0, 0}$ and 
$\OO_{B_0, r}$ is not a ring unless $r=0$. 
\par 
For integers $n \geq 1$, $r \geq 0$ and a real number $\varepsilon >0$, 
we consider the   Diophantine equation 
\begin{equation}
\label{e:unit-equation-geometric-general}
x+ y = z
\end{equation}
with $(x, y, z) \in K^3$ satisfying the following conditions:    
\begin{enumerate} [\rm (i)]
\item
$x \in \OO_S^*(B_0, \varepsilon)$; 
\item
$y^{-1} \in \OO_{B_0, r}$; 
\item
$z \in L(nD_S) \setminus \{0\} = \{ h \in K^* \colon \divisor (h) + nD_S \geq 0 \}\subset K^*$.  
\end{enumerate}

In other words, we consider the union of 
solutions $(x, y) \in \OO_S^*(B_0, \varepsilon) \times (\OO_{B_0, r} \setminus \{0\})^{-1}$ 
of the parametrized equations 
$x+y=z$ with $z$ varying in the space $\in L(nD_S) \setminus \{ 0\}$.  
 \par
In the case $r= \varepsilon=0$ and $B_0=B \setminus S$, we recover the usual $S$-unit equation 
$x+y=1$ with $x, y \in \OO_S^*$ by setting $z=1 \in L(nD_S) \setminus \{0\}$. 
Indeed,   $r=0$ and $B_0=B \setminus S$ imply $\OO_{B_0, r}= \OO_S$. 
On the other hand, if $x \in \OO_S^*$ and $y^{-1} \in \OO_S$ such that $x+y = 1$, 
then $y= 1- x \in \OO_S$ since $\OO_S$ is a ring and thus $y \in \OO_S^*$. 
Therefore,   \eqref{e:unit-equation-geometric-general} generalizes 
the usual $S$-unit equation over function fields. 
  \par
 The finiteness of the numbers of solutions $x,y \in \OO_S^*$ with $x/y \notin \C^*$ of the unit equation $x+y=1$ 
is well-known. 
It turns out that 
Theorem \ref{t:s-unit-equation-geometric-general} actually implies that a similar  property 
still holds for the generalized equation \eqref{e:unit-equation-geometric-general}. 

\begin{corollary}
\label{c:s-unit-equation-geometric-general-1}
There are only finitely many $x \in \OO_S^*(B_0, \varepsilon)$ modulo $\C^*$ such that the equation \eqref{e:unit-equation-geometric-general} 
admits a solution.    
\end{corollary}

\begin{proof}[Proof of Corollary \ref{c:s-unit-equation-geometric-general-1}] 
Denote $d= \dim L(nD_S) = \dim H^0(B, \OO(nD_S))$. 
Fixing a basis $(z_1, \dots, z_d)$ of the complex vector space $L(nD_S)$,  
we consider the compact unit sphere 
$$
\mathbf{S}^d = \left \{   \sum_{k=1}^d a_k z_k \in L(nD_S) \colon || (a_1, \dots, a_k) || = 1 \right \} \subset L(nD_S) 
$$ 
where $ || (a_1, \dots, a_k) || = \left (\sum_{k=1}^d a_k \right)^{1/2}$. 
Recall that $D_S \coloneqq \sum_{ b \in S} [b]$.   
\par
Define $\tilde{Z} \coloneqq  L(nD_S) \setminus \{0\} \simeq \C^d\setminus \{0\}$ then 
$\tilde{Z}$ is a integral smooth algebraic variety. 
We have a canonical valuation morphism 
$$
\val \colon  B \times \tilde{Z} \to \Proj^1 , \quad (b, z) \mapsto z(b). 
$$
Consider the flat family of divisors 
$\DD \subset X \times \tilde{Z} $ given by the 
image of the algebraic section 
$$
\Sigma \colon B \times \tilde{Z} \to \Proj^1 \times B \times \tilde{Z}, 
\quad (b, z) \mapsto  (z(b), b,z). 
$$
For each $z \in \tilde{Z} \subset K^*$, 
let $(z) \subset X$ be the induced section of the projection 
$f \colon X \to B$. 
It is clear that $(z)= \DD_z \not \subset (0) \cup (\infty)$ for every $z \in \tilde{Z}$ 
by the construction of $\DD$. 
Now let $z \in \tilde{Z} \subset K^*$, $x \in \OO_S^*(B_0, \varepsilon)$ and 
$y=z-x$. 
It is not hard to see that 
the condition $y^{-1} \in \OO_{R}$ for a certain subset $R \subset B$ verifying $\# (R \cap B_0) \leq r$ 
means exactly that 
\begin{align}
\label{r:coroolary-s-unit-equation-geometric-first}
x \in   \cup_{R \subset B, \# R \cap B_0 \leq r}  
\{ x' \in \OO_S^*(B_0, \varepsilon) \colon x' \text{ is } (R,\DD_z)\text{-integral in } X   \}. 
\end{align}
We cannot apply directly 
Theorem \ref{t:s-unit-equation-geometric-general} 
since $\tilde{Z}= L(nD_S) \setminus \{0\}$ is not compact. 
However,  
it suffices to restrict ourselves to the case $z \in Z \coloneqq \mathbf{S}^d$ since 
the compact subspace $\mathbf{S}^d$ 
contains all classes modulo $\C^*$ of $\tilde{Z}$. 
 Therefore,   Theorem \ref{t:s-unit-equation-geometric-general} says that  the set 
$$
J_{r, \varepsilon} = \cup_{z \in \mathbf{S}^d}  \cup_{R \subset B, \# R \cap B_0 \leq r}  
\{ x' \in \OO_S^*(B_0, \varepsilon) \colon x' \text{ is } (R,\DD_z)\text{-integral in } X   \}
$$
is finite. 
Combining with \eqref{r:coroolary-s-unit-equation-geometric-first},  
the proof of Corollary \ref{c:s-unit-equation-geometric-general-1} is completed. 
\end{proof}

Following the general idea that parametrized Diophantine equations have 
no or very few  
integral solutions under a general choice of parameters, 
we mention below a remarkable theorem on unit equations in the case of number fields. 

\begin{theorem} 
[Evertse-Gy\" ory-Stewart-Tijdeman]
\label{t:unit-atmost2}
Let $K$ be a number field and $S$ a finite number of places. 
There exists a finite set of triples $A \subset (K^*)^3$ with the following property. 
For every  
$\alpha=(\alpha_1, \alpha_2, \alpha_3) \in (K^*)^3$ whose class $[\alpha] \in (K^*)^3/(K^*(\OO_S^*)^3)$ does not belong to 
$[A] \subset (K^*)^3/(K^*(\OO_S^*)^3)$,  
the $S$-unit equations 
\begin{align}
\label{e:unit}
 \alpha_1 x+\alpha_2y=\alpha_3
 \end{align} 
has at most 2 solutions. 
\end{theorem} 

\begin{proof}
See \cite[Theorem 1]{evertse-gyory-stewart-tijdeman}. 
Note that the natural action of $K^*(\OO_S^*)^3$ on $(K^*)^3$ is given by: 
$(c, (u,v,w)) \cdot (\alpha_1, \alpha_2, \alpha_3) = (cu\alpha_1, cv\alpha_2, cw \alpha_3).$
\end{proof}

 Theorem \ref{t:unit-atmost2} implies that almost all equations 
of the form \eqref{e:unit} have no more than $2$ unit solutions. 
As an analogous result for certain Diophantine equations in function fields, 
Corollary \ref{c:s-unit-equation-geometric-general-1} can be directly reformulated 
as follows. 

\begin{corollary}
\label{c:few-solution-generalized-unit}
Given $\varepsilon > 0$, let $\omega \in \OO_S^*(B_0, \varepsilon)$ and $n \geq 1$. 
Consider the equation 
\begin{equation}
\label{e:analogue-few-sol-unit-function-field}
x+y= \omega
\end{equation}
with unknowns $x,y \in K^*$ satisfying $x \in L(nD_S) \setminus \{0\}$ and $y^{-1} \in \OO_{B_0, r}$ 
(cf. \eqref{e:definition-o-b-r}). 
There exists a finite subset $A \subset \OO_S^*(B_0, \varepsilon)$ such that 
whenever  $\omega \notin \C^* A$, 
 the equation \eqref{e:analogue-few-sol-unit-function-field} has no solutions. 
Moreover, there exists $m >0$ such that we can take $A$ having no more than 
$m(r+1)^{2\rank{\pi_1 (B_0)}}$ elements for every $r \in \N$.  
\end{corollary}
Note that the above last statement follows immediately from Theorem  \ref{t:s-unit-equation-geometric-general}.

\subsection{Proof of Theorem  \ref{t:s-unit-equation-geometric-general}}
\label{s:proof-of-ruled-surface}
\subsubsection{Preliminaries} 
 We begin with 
an easy analogue of the Lang-N\' eron theorem 
for the multiplicative group $\G_m$. 
Recall that $S$ is a finite subset of $B$. 

\begin{proposition}
\label{p:integer-ring-finite-rank}
$\OO_S^*/\C^*$ is a torsion-free abelian group of rank $\leq \# S$.  
\end{proposition} 

\begin{proof}
 Consider the following homomorphism of groups 
$$
\rho \colon \OO_S^* \to \oplus_{\nu \in S} \Z , \quad f \mapsto  (\mult_\nu( \divisor f) )_{\nu \in S}.  
$$
We claim that $\Ker \rho = \C^*$. 
Indeed, suppose that $f \in \OO_S^*$ satisfies $\rho (f) =0$. 
Since $f \in \OO_S^*$, all poles and zeros of $f$ belongs to $S$. 
However, $\rho (f)=0$ implies that these poles and zeros are all of order $0$. 
It follows that the corresponding morphism $f \colon B \to \Proj^1$ must be constant 
and thus $f \in \C^*$ as claimed. 
Therefore, we have an injective homomorphism of groups 
$\OO_S^*/ \C^* \to \oplus_{\nu \in S} \Z$.  
 $\OO_S^*/\C^*$ is a torsion-free abelian group of rank $\leq \# S$.  
\end{proof}

Now each $z \in \OO_S^*= \G_m(\OO_S)$ induces a section $\sigma_z$ of the projection $Y \to B_0$ and thus 
a section $i_z$ of the following exact sequence of fundamental groups: 
\begin{equation}
\label{e:exact-sequce-unit-1}
0 \to \pi_1(Y_{b_0}, w_0) \to \pi_1(Y, w_0) \xrightarrow{\eta}  \pi_1(B_0, b_0) \to 0, 
\end{equation}
where we fix $w_0 =1 \in Y_{b_0}= \C^*$ above 
the fixed point $b_0 \in B_0$. 
\par
Fix a collection of geodesics  $l_{w_0,w} \colon [0,1] \to Y_{b_0}$ 
on $Y_{b_0}$ 
such that $l_{w_0,w} (0)=w_0$ and $l_{w_0,w} (1)=w \in Y_{b_0}$. 
Every $x \in \OO_S^*$ induces a section $\sigma_x \colon B_0 \to Y_{B_0}$ which in turn gives rise to 
a section $i_x \colon \pi_1(B_0,b_0) \to \pi_1(Y_{B_0}, w_0)$ of the exact sequence 
\eqref{e:abelian-homotopy-exact-sequence-1} as follows. 
For every loop $\gamma$ of $B_0$ based at $b_0$, we define: 
\begin{equation}
\label{e:definition-of-i-p-main-reduction-step}
i_x([\gamma])=[ l^{-1}_{w_0,\sigma_x(b_0)} \circ \sigma_x(\gamma) \circ l_{w_0,\sigma_x(b_0)}] 
\in  \pi_1(Y_{B_0}, w_0). 
\end{equation} 
 \par 
 Denote $G= \pi_1(B_0, b_0)$ and let 
$\widehat{G}$ be the profinite completion of $G$.  
 As in the case of elliptic surfaces (Proposition \ref{p:homotopy-rational-abelian}), we have the following reduction result.  

\begin{theorem}
\label{t:homotopy-northcott-unit-equation}
Let $n \geq 2$ be an integer. 
We have the following commutative diagram 
of homomorphisms of groups: 
 \[ 
\label{d-trivial-diagram}
\begin{tikzcd}
\OO_S^*/(\OO_S^*)^n  \arrow[r, hook , "\delta"] 
 &  H^1(\widehat{G}, \mu_n) 
   \arrow[dr, "\simeq"]
    & \\
\OO_S^* \arrow[u] \arrow[r, "\alpha"] 
& H^1(G, \Z) \arrow[r, "\beta"] 
& H^1(G, \mu_n). 
\end{tikzcd}
\]
Moreover, two elements of $\OO_S^*$ induces the same section of the exact sequence of 
fundamental groups \eqref{e:exact-sequce-unit-1}
if and only if they differs by a factor in $\C^*$.  
\end{theorem}

The proof of Theorem \ref{t:homotopy-northcott-unit-equation} 
will be given in Appendix \ref{appendix-unit-diagram}. 

\subsection{Key lemma} 

Recall that $d_M$ means the Kobayashi hyperbolic pseudo-metric
on the complex space $M$ 
and $\rho$ is a fixed Hermitian metric on the smooth surface $X$. 
As in the case of elliptic fibrations, 
the main additional ingredient in the proof of Theorem \ref{t:s-unit-equation-geometric-general} for ruled surfaces is  the following 
analogous technical lemma of Lemma \ref{l:key-lemma-strong-uniform-elliptic}:

\begin{lemma} 
\label{l:key-lemma-strong-unit-geometric-general} 
 Let $\varepsilon > 0$. 
Then there exists $M > 0$ such that for each $z_i \in \tilde{Z}$, 
we have the following data:
\begin{enumerate} [\rm (a)]
\item
an analytic open neighborhood $U_i$ of $z_i$ in $\tilde{Z}$;
\item
a disjoint union $V_i$ consisting of $\leq M$ discs each of $\rho$-radius $\leq \varepsilon$ in $B_0$; 
\item
a constant $c_i > 0$;
\end{enumerate}
with the following property: 
\begin{enumerate} [\rm (Q)]
\item 
for each $z \in U_i$, we have 
$
d_{(Y \setminus \DD_z) \vert_{(B_0 \setminus V_i)}} \geq c_i \rho\vert_{(Y \setminus \DD_z) \vert_{(B_0 \setminus V_i)}}. 
$
\end{enumerate}

\end{lemma}

The proof of Lemma \ref{l:key-lemma-strong-unit-geometric-general} applies, \emph{mutatis mutandis}, 
the proof of  Lemma \ref{l:key-lemma-strong-uniform-elliptic}  with some minor modifications. 
The same remark at the beginning of the proof of Lemma \ref{l:key-lemma-strong-uniform-elliptic} 
shows that there exists $N' > 0$ such that the total
number of irreducible components (counted with multiplicities) of each effective divisor $\DD_z$ 
is at most $N'$ for every $z \in  Z$. 
\par
Moreover, since $\DD \to B \times \tilde{Z}$ is flat, every divisor 
$\DD_z$, $z \in \tilde{Z}$, 
contains no vertical components with respect to the projection $f \colon X \to B$. 
The divisors $\DD_z$ are   numerically equivalent and are not contained in the curve $(0)\cup (\infty) \subset X$. 
In particular, it follows that  for some constant $N'' > 0$, we have 
$$
\# \DD \cap ((0) \cup (\infty)) \leq N'', \quad \text{ for all }z \in Z.
$$ 

By the adjunction formula,  
$$
p_1 \coloneqq p_a(\DD_z) = \frac{\DD_z(\DD_z+K_X)}{2} +1 \geq 0
$$
is a constant independent of $z \in \tilde{Z}$. 
Let $T'_z \subset B$ be the image in $B$ of the union of the ramification points of the ramified cover
of algebraic curves $\pi_z \colon (\DD_z)_{red} \to  B$, 
and of the singular points of $(\DD_z)_{red}$. 
We have as in the relation \eqref{e:key-lemma-strong-uniform-elliptic-1} that:    
$$
\# T'_z  \leq M' \coloneqq 2N' + 2(p_1 -2), \quad \text{ for all } z \in \tilde{Z}. 
$$
Define $T_z = T'_z  \cup f( \DD \cap ((0) \cup (\infty)) ) \subset B$ then it follows from the above discussion that 
\begin{equation}
\label{e:key-lemma-unit-equation-geometric-general} 
\#T_z  \leq M \coloneqq M' + N'', \quad \text{ for all } z \in \tilde{Z}. 
\end{equation}

We return to the proof of Lemma \ref{l:key-lemma-strong-unit-geometric-general}. 
Again, the idea of the proof is the same as in Lemma \ref{l:key-lemma-strong-uniform-elliptic} 
 but we indicate in details the needed modifications. 

\begin{proof}[Proof of Lemma \ref{l:key-lemma-strong-unit-geometric-general}] 
Fix $\varepsilon >0$   and $z_i \in \tilde{Z}$. 
Recall that $B_0 = B \setminus (U \cup V)$.  
Let $M$ be the constant defined in \eqref{e:key-lemma-unit-equation-geometric-general}. 
We can clearly choose a finite disjoint union $V_i$ of at most $M$  
nonempty closed discs in $B_0$ of $\rho$-radius $\leq \varepsilon$ 
to cover the points of $T_{z_i} \setminus  V$. 
 Thus, we obtain the data $V_i$ for Lemma \ref{l:key-lemma-strong-unit-geometric-general}.(b). 
\par
Define $\tilde{\DD}= \DD \cup ( ( (0) \cup (\infty)) \times \tilde{Z})  \subset X\times \tilde{Z}$ 
and for each $z\in \tilde{Z}$, let 
$$
Y_z \coloneqq (Y \setminus \DD_z)\vert_{B_0 \setminus V_i} 
= (X \setminus \tilde{\DD}_z)\vert_{B_0 \setminus V_i} \subset X. 
$$ 
To show the existence of  $U_i$, $c_i$ satisfying (i), we suppose the contrary. 
By the continuity of $\rho$ on the compact unit
tangent space of $X$, there would exist a sequence $(z_n)_{n \geq 1} \subset \tilde{Z}$ 
such that $z_n \to z_i$ in the
analytic topology and a sequence of holomorphic maps
$$
h_n \colon  \Delta_{R_n} \to  Y_{z_n}
$$
such that $R_n \to \infty$ 
and $|dh_n(0)|_\rho > 1$. 
Here, $\Delta_{R} \subset \C$ denotes the open disc of radius $R$ in the complex plane.  
By Brody's reparametrization lemma (Theorem \ref{l:brody-reparametrization}), we obtain a sequence of holomorphic maps
$$
g_n \colon \Delta_{R_n} \to Y_{z_n} 
$$
such that $|dg_n(0)|_\rho = 1$ and $|dg_n(z)|_\rho \leq 
 R^2_n/ (R_n^2 - |z|^2)$ 
for all $z \in \Delta_{R_n}$. 
In particular, $|dg_n(z)|_\rho \leq  4/3$ for all $z \in \Delta_{R_n/2}$. 
It follows that the family $(g_n)_{n\geq 1}$ is equicontinuous
with image inside the compact space $X\vert_{B_0 \setminus V_i}$. 
Up to passing to a subsequence,  
$(g_n)_{n \geq 1}$ converges uniformly on compact subsets of $\C$ to
a holomorphic map $g \colon \C \to  X\vert_{B_0 \setminus V_i}$
with $|dg(0)|_{\rho} = 1$. 
 \par
Since $B_0 \setminus  V_i$ is  hyperbolic,  
$f \circ g \colon \C \to  B_0 \setminus  V_i$ is a constant $b_* \in B_0 \setminus  V_i$. 
As in Lemma \ref{l:key-lemma-strong-uniform-elliptic}, 
we can apply Lemma \ref{l:analytic-cover-deformation} to find a constant $\lambda \in \N$, 
a small analytic open disc $\Delta \subset B_0 \setminus V_i$ containing $b_*$  
and a small analytic connected open neighborhood $U_i$ of $z_i$ in $\tilde{Z}$ 
with: 
$$
\tilde{\DD}\vert_{\Delta \times U_i} \to \Delta \times U_i
$$ 
an \' etale cover of degree $\lambda$. 
Shrinking $\Delta$ and $U_i$ if necessary, 
$\tilde{\DD}\vert_{\Delta \times U_i}$ consists of $\lambda$ disjoint connected components including 
$(0)\vert_\Delta \times U_i$ and 
$(0)\vert_\Delta \times U_i$. 
\par
Fixing a connected component $\DD^0$ of $\DD \vert_{\Delta \times U_i}$, 
we obtain a family of $1$-sheeted covers $D_{z}^0 \to \Delta$ with $z \in U_i$ satisfying 
$D_z^0 \cap ((0) \cup (\infty))= \varnothing$.   
Hence, each $D_{z}^0$ for $z \in U_i$ is a complex submanifold of $Y_\Delta$ via the inclusion $\iota_z \colon D_{z}^0 \to Y_\Delta$. 
The composition $f \circ \iota_z \colon D_{z}^0 \to \Delta$ is thus holomorphic and bijective for every $z \in U_i$.  
As injective holomorphic maps are biholomorphic to their images (\cite[Theorem 2.14, Chapter I]{range-86-book}), 
the map $s_{z} \colon \Delta \to D_{z}^0 \to Y_\Delta$ given by $t \mapsto D_{z}^0(t)= f^{-1}(t) \cap D_{z}^0$ 
is holomorphic for every $z \in U_i$. 
Moreover, the map 
\begin{align*}
\Delta \times U_i \to Y_\Delta \times U_i, \quad 
(t, z) \mapsto (D_{z}^0(t) ,  z)= (f^{-1}(t) \cap D_{z}^0 , z) 
\end{align*}
is holomorphic. 
As $Y_\Delta=F \times \Delta$, we can   write 
$s_z(t)=(u(t, z), t)$ where $u \colon \Delta \times U_i \to \C^*$ is a holomorphic function. 
 The maps $\Psi_{z} \colon Y_\Delta \to Y_\Delta$, $z \in U_i$, 
given by the fibrewise multiplication by $s_{z} s_{z_i}^{-1}$, 
i.e., 
$$
\Psi_z(x, t)= (x u(t,z) u(t, z_i)^{-1}, t) , \quad (x, t) \in Y_\Delta = \C^* \times \Delta, 
$$ 
  form a smooth family of biholomorphisms commuting with the projection $ Y_\Delta \to \Delta$. 
 \par
By passing to a subsequence with suitable restrictions of domains of definitions, we can assume 
that $\im(g_n) \subset X_\Delta$ 
and that the holomorphic maps $g_n \colon \Delta_{R_n} \to X_\Delta \setminus \tilde{\DD}_{z_n}$ still satisfy 
$$
R_n \to \infty , \quad |dg_n(0)|_\rho = 1. 
$$
Since $D_{z}^0 \subset \DD_z$, we can consider the sequence of holomorphic maps
$$
f_n \coloneqq \psi_{z_n} \circ g_n \colon \Delta_{R_n} 
\to  Y_\Delta\setminus D_{z_i}^0= X_\Delta \setminus (D_{z_i}^0 \cup (0) \cup (\infty))
$$
into the fixed space $ Y_\Delta\setminus D_{z_i}^0$. 
By the smoothness of the family of biholomorphisms
$(\psi_z)_{z \in U_i}$ 
 and the compactness of the $\rho$-unit tangent bundle of $X$, 
there exists   $c > 1$ such that 
\begin{equation}
\label{e:contradiction-unit}
c^{-1}\leq |df_n(0)|_\rho \leq c, \quad \text{ for all } n \geq 1. 
\end{equation}
\par
Now, consider any holomorphic map $h \colon \C \to Y_{\Delta} \setminus D_{z_i}^0$. 
The composition $f \circ h$ must be a constant since $\Delta$ is hyperbolic. 
Thus, $h$ factors through a fibre of $X_{\Delta} \setminus (D_{z_i}^0 \cup (0) \cup (\infty))$. 
However, as each such fibre is the complement of at least 3 points in $\Proj^1$ and thus is hyperbolic, 
$h$ must be constant. 
Similarly, it is clear that each holomorphic map $\C \to (D_z^0 \cup (0) \cup (\infty))\vert_{\bar{\Delta}}$ is constant. 
 Green's theorem \ref{t:green-hyperbolic-embedding} implies that  
$Y_{\Delta} \setminus \DD_{z_i}$ is hyperbolically
embedded in $X_{\Delta}=f^{-1}(\Delta)$ 
(with respect to the metric $\rho\vert_{X_\Delta}$).  
But since $R_n \to \infty$, we clearly obtain a contradiction using  
\eqref{e:contradiction-unit} and  
Remark \ref{r:infinitesimal-hyperbolic}.  
We have therefore proved the existence of the data in (a), (b) and (c) such that (Q) is satisfied. 
\end{proof}

\subsection{Proof of Theorem \ref{t:s-unit-equation-geometric-general}}

We can now return to the main result which 
will be very similar to the proof of Theorem \ref{t:strong-uniform-elliptic}. 
Let $(\alpha_1, \dots, \alpha_k)$ be a fixed system of generators 
of the fundamental group $\pi_1(B_S, b_0)$ 
with a fixed based point $b_0 \in B_0$. Let $w_0=1 \in Y_b$.  

\begin{proof} [Proof of Theorem \ref{t:s-unit-equation-geometric-general}] 

Fix $\varepsilon > 0$. 
We can enlarge slightly the discs in $V$ 
if necessary. 
Then we obtain a constants $M > 0$ 
and a set of data $(U_i,  V_i,  c_i)$ 
for each element $z_i \in Z$ as in Lemma \ref{l:key-lemma-strong-unit-geometric-general}. 
Consider the open covering $Z \subset \cup_{z_i \in Z} U_i$. 
Since $Z$ is compact, there exists 
$Z_* \subset Z$ finite such that $Z \subset \cup_{z_i \in Z_*} U_i$. 
As $c_i > 0$ for every $z_i \in Z$ by Lemma \ref{l:key-lemma-strong-unit-geometric-general}, 
we have  
\begin{equation}
\label{e:c-star-unit-strong-proof-main}
c_* \coloneqq \min_{z_i \in Z_*} c_i > 0. 
\end{equation}
\par
For each $z_i \in Z_*$,  denote by $L_i>0$ the maximum of 
the constants given by Theorem \ref{t:linear-bound-s-base-curve-1} applied to  
$B_i \coloneqq B \setminus (V \cup V_i)= B_0 \setminus V_i$ and to each free homotopy classes  
$\alpha_1, \dots, \alpha_k$ regarded as elements of $\pi_1(B_i) \supset \pi_1(B_0)$. 
Let 
\begin{equation}
\label{e:l-max-unit-strong-proof-main}
L= \max_{z_i \in \Z_*} L_i >0. 
\end{equation}
 \par
Now let $x \in J_{r, \varepsilon}$, that is, $x \in \OO_S^*(B_0, \varepsilon)$ 
is $(R, \DD_z)$-integral in $X$ for some $z \in Z$,  $R \subset B$ 
such that $\# R \cap B_0 \leq r$. 
As $Z  \subset \cup_{z_i \in Z_*} U_i$, there exists $z_i \in Z_*$ 
such that $z \in U_i$.

By Theorem \ref{t:linear-bound-s-base-curve-1} applied 
to $B_i $, 
there exists $b_i \in B_i$ and a system of loops 
$\gamma_1, \dots,  \gamma_k$ based at $b$
 representing respectively the homotopy classes $\alpha_1, \dots, \alpha_k$ 
 up to a single conjugation 
such that  $\gamma_j \subset B_i \setminus R$ for every $j=1, \dots, k$ and that 
\begin{equation}
\label{e:strong-uniform-unit-proof-main} 
\length_{d_{B_i \setminus R}} (\gamma_j) 
\leq L_i( \# R \cap B_i + 1) \leq L_i( r + 1).
\end{equation} 
The second inequality follows from $ \# R \cap B_i \leq  \# R \cap B_0 \leq r$. 
\par
Let $\sigma_x \colon B \to X$ 
be the  section  induced by $x$. 
For every $j \in \{1, \dots, k\}$, we find that 
 $\sigma_x(\gamma_j) \subset 
 (Y \setminus \DD_z)\vert_{B_i \setminus R} \subset (Y \setminus \DD_z)\vert_{B_i}$ 
 by the definition of $(R, \DD)$-integral points and because $x \in \OO_S^*$. 
 It follows that:
\begin{align*}
\length_\rho (\sigma_x(\gamma_j)) 
& \leq c_i^{-1} \length_{d_{(Y \setminus \DD_z)\vert_{B_0 \setminus V_i}}} 
(\sigma_x (\gamma_j)) 
&  \text{(by Lemma }\ref{l:key-lemma-strong-unit-geometric-general}) 
\\
& 
 = c_i^{-1} \length_{d_{(Y \setminus \DD_z)\vert_{B_i}}} 
(\sigma_x (\gamma_j)) 
& (\text{as } B_i\coloneqq   B_0 \setminus V_i \subset B_0) 
\\
& \leq 
 c_i^{-1} \length_{d_{(Y \setminus \DD_z)\vert_{B_i \setminus R}}} 
(\sigma_x (\gamma_j)) 
& \text{(as } ( Y \setminus \DD_z)\vert_{B_i \setminus R} \subset (Y \setminus \DD_z)\vert_{B_i})  
\\
& \leq 
c_*^{-1} \length_{d_{B_i\setminus R}} (\gamma_j) 
& (\text{by } \eqref{e:c-star-unit-strong-proof-main} \text{ and Lemma } \ref{l:section-geodesic})
\\
& \leq  c_*^{-1}L( r +1).  
&   (\text{by }  \eqref{e:l-max-unit-strong-proof-main} \text{ and } \eqref{e:strong-uniform-unit-proof-main}) 
\end{align*}
Let $\delta$ be the $\rho$-diameter of $X $. 
Then the
homotopy section $i_x$ 
associated to $x$ of the short exact sequence (cf. the sequence \eqref{e:exact-sequce-unit-1} in Theorem \ref{t:homotopy-northcott-unit-equation})  
\begin{equation*}
0 \to  \pi_1(Y_{b_0}, w_0) \to \pi_1(Y_{B_0}, w_0) \to \pi_1(B_0, b_0)\to 0
\end{equation*}
sends the basis $(\alpha_j)_{1 \leq j \leq k}$ of $\pi_1(B_0, b_0)$ 
to the classes in $\pi_1(Y_{B_0}, w_0)$ 
which admit representative loops of $\rho$-lengths bounded by 
$H(r) \coloneqq c_*^{-1}L(r+1) +  2\delta$. 
\par
 As of Theorem \ref{t:strong-uniform-elliptic}, the rest of the proof follows tautologically the same lines 
 of the proof of \cite[Theorem A]{phung-19-abelian}. 
We can thus obtain a finite number $m >0 $ such that: 
\begin{equation*}
\# (J_{r , \varepsilon} \text{ mod } \C^*) \leq m(r+1)^{2 \rank \pi_1(B_0)}, \quad \text{for every } r \in \N. 
\end{equation*} 

The only needed modification is the following. 
We can clearly assume  that $0 < \varepsilon <1$. 
Consider the compact bordered manifold 
$E_\varepsilon  \coloneqq \{ z \in \C \colon \varepsilon \leq |z| \leq  \varepsilon^{-1} \}$.  
Remark that $\pi_1(E_\varepsilon, w_0)= \Z$ (recall that $w_0=1 \in Y_{b_0}= \C^*$).  
\par
Since $x \in \OO_S^*(B_0, \varepsilon)$,
it actually induces a homotopy section $i_{x , \varepsilon}$ 
of the short exact sequence   
\begin{equation*}
0 \to  \pi_1(E_\varepsilon, w_0) \to \pi_1(E_\varepsilon \times B_0, w_0) \to \pi_1(B_0, b_0)\to 0 . 
\end{equation*}
Since $\pi_1(E_\varepsilon \times B_0, w_0)= \pi_1(E_\varepsilon, w_0)  \times \pi_1(B_0, b_0)$, 
the homotopy section 
$i_{x, \varepsilon}$ is determined by the $\pi_1(E_\varepsilon, w_0)$-component 
of $i_{x, \varepsilon} (\alpha_j)$ for every $j= 1, \dots, k$. 
But we have shown above that the induced $\rho$-length of certain representative loops of these components 
are bounded by $H(r) \coloneqq c_*^{-1}L(r+1) +  2\delta$. 
The representative loops are in fact $ \mathrm{pr}_1 (\sigma_x(\gamma_j))$ 
where $\mathrm{pr}_1 \colon E_\varepsilon \times B_0 \to E_\varepsilon$ is the first projection. 
\par
As $E_{\varepsilon}$ is compact and $\pi_1(E_\varepsilon, w_0)= \Z$, we can thus conclude 
as in the proof of \cite[Theorem A]{phung-19-abelian} by applying \cite[Lemma 13.10]{phung-19-abelian}. 
The proof is thus completed. 
\end{proof}

\section{Appendix: Proof of Theorem \ref{t:homotopy-northcott-unit-equation}}
\label{appendix-unit-diagram}
 
Recall that $Y= \C^* \times B_0$. 
 We regard  
$f \colon Y \to B_0$ as a constant sheaf of (multiplicative) 
abelian groups 
over $B_0$. 
There is a canonical short exact sequence of sheaves over $B_0$ 
induced by the exponential map: 
\begin{equation}
\label{e:exp-relative} 
0 \to (R^1f_* \Z )^\vee \to T_Y \to Y \to 0. 
\end{equation}
 
Consider the $n$-th power $B_0$-morphism 
$\wedge n \colon Y \to Y$. 
The induced differential map 
$d(\wedge n) \colon   T_Y \to T_Y$ 
is an isomorphism since we are in characteristic $0$ so that $n$ is invertible. 
\par
The group of global algebraic sections $Y(B_0)$ is naturally identified 
with a subgroup of $Y_K(K)=K^*$. 
Clearly, $Y[n]\coloneqq \Ker( \wedge n)$ is the constant sheaf $\mu_n$ on $B_0$. 
\par
The map $\wedge n  $ and the sequence \eqref{e:exp-relative} 
induce a commutative digram: 

\begin{equation}
\label{d-abelian-family-diagram-general}
\begin{tikzcd}
 &        
&  
&   0 \arrow[d] 
& \\
 & 0 \arrow[r]  \arrow[d] 
 & 0 \arrow[d] \arrow[r] & Y[n]   \arrow[d]  &   \\ 
0 \arrow[r] &  (R^1f_* \Z )^\vee  \arrow[d, "n"] \arrow[r] 
 & T_Y \arrow[r] \arrow[d, " d( \wedge n) "]  
 & Y \arrow[d, "\wedge n"]  \arrow[r] &  0  \\ 
0 \arrow[r] &  (R^1f_* \Z )^\vee  \arrow[d] \arrow[r] 
 & T_Y \arrow[r] \arrow[d,]  
 & Y \arrow[r] \arrow[d]  &  0 \\ 
  & Q \arrow[d]       \arrow[r] 
& 0 \arrow[r] 
& 0  &  
\\
& 0 & & &
\end{tikzcd}
\end{equation}

We have a natural isomorphism 
$Y[n] \simeq Q$ by the snake lemma. 
The cohomology long exact sequences induced by 
Diagram \eqref{d-abelian-family-diagram-general} give a natural commutative diagram:  

 \begin{equation}
\label{d-abelian-family-diagram-general-2}
\begin{tikzcd}     
 Y(B_0)   \arrow[d, "n"] \arrow[r] 
 & H^1(B_0, (R^1f_* \Z )^\vee)  \arrow[r] \arrow[d, "n "]  
 & H^1(B_0, T_Y)  \arrow[d, "\simeq"]  \\ 
 Y(B_0)    \arrow[d ] \arrow[r ] 
 & H^1(B_0, (R^1f_* \Z )^\vee)  \arrow[r] \arrow[d]  
 & H^1(B_0, T_Y) \arrow[d]  \\ 
  H^1(B_0,  Y[n])      \arrow[r, "\simeq"] 
& H^1(B_0, Q) \arrow[r] 
& 0    
\end{tikzcd}
\end{equation}

We decompose the map $Y(B_0) \to  H^1(B_0,  Y[n])$ in the first column as: 
$$
Y(B_0) \to Y(B_0) / Y(B_0)^n \hookrightarrow H^1(B_0,  Y[n]). 
$$
Hence, we obtain a natural commutative diagram from \eqref{d-abelian-family-diagram-general-2}: 

 \begin{equation}
\label{d-unit-family-diagram-1}
\begin{tikzcd}
\OO_S^*/(\OO_S^*)^n  \arrow[r, hook]  
&Y(B_0)/ Y(B_0)^n 
 \arrow[r, hook ]   
  &  H^1(B_0, \mu_n)  \arrow[dr, "\simeq"] &  \\ 
\OO_S^*  
\arrow[u]  \arrow[r, hook] 
& Y(B_0) 
\arrow[u] \arrow[r ] 
& H^1(B_0, (R^1f_* \Z )^\vee) \arrow[r ] 
& H^1(B_0, Q). 
\end{tikzcd}
\end{equation}

Since $Y$ and $B_0$ are $K(\pi,1)$-spaces, 
 the cohomology of local systems coincide with the group cohomology. 
 We thus have canonical isomorphisms: 
$$
 H^1(B_0, \mu_n) \simeq H^1(G, \mu_n), \quad H^1(B_0,(R^1f_* \Z )^\vee)  \simeq H^1( G, \Gamma), 
$$ 
where $G= \pi_1(B_0, b_0)$ and $\Gamma = H_1(Y_{b_0}, \Z) \simeq (R^1f_* \Z )^\vee_{b_0}\simeq \Z$. 
The actions of the group $G$ on $\Gamma$ and on $\mu_n$ are trivial (the monodromy is trivial 
in a trivial fibration). 
By \cite[I.2.6.b]{serre:galois-cohomology}, 
there is a natural isomorphism $H^1(\widehat{G}, \Gamma) \simeq H^1(G,  \Gamma)$ 
induced by the injection $G \to \widehat{G}$.  
Hence, \eqref{d-unit-family-diagram-1} implies now easily 
the commutative diagram 
in Theorem \ref{t:homotopy-northcott-unit-equation}. 
\par
The detailed descriptions of the involved homomorphisms 
can be found in \cite{phung-19-phd}.  
\par
The last statement of Theorem  \ref{t:homotopy-northcott-unit-equation} 
follows directly from Proposition \ref{p:integer-ring-finite-rank} 
and the fact that the map $\alpha \colon \OO_S^* \to H^1(G, \Z)$ 
is a homomorphism of groups defined by $\alpha(x)=i_x - i_1$ (cf. \eqref{e:definition-of-i-p-main-reduction-step}).

\bibliographystyle{siam}

\end{document}